\documentclass{article}
 
\usepackage{amsthm}
\usepackage{amsfonts}
\usepackage{amssymb}
\usepackage{verbatim}
\usepackage[all]{xy}
\usepackage{amsmath, amscd, amsthm}
\usepackage{longtable}
\usepackage{amscd}
\usepackage{mathrsfs}
\usepackage{graphicx}
\usepackage{latexsym}

\newtheorem{theorem}{Theorem}[section]
\newtheorem{lemma}[theorem]{Lemma}

\newtheorem{cor}[theorem]{Corollary}
\newtheorem{ex}[theorem]{Example}
\newtheorem{defn}[theorem]{Definition}
\newtheorem{remark}[theorem]{Remark}

\usepackage{latexsym}

\begin{document}

\title{Length Functions for Semigroup Embeddings}
\author{Tara Davis}

\maketitle

\begin{abstract}
Following the work done in \cite{olshanskii} for groups, we describe, for a given semigroup $S$, which functions $l : S \rightarrow \mathbb{N}$ can be realized up to equivalence as length functions $g \mapsto |g|_{H}$ by embedding $S$ into a finitely generated semigroup $H$. We also, following the work done in \cite{olsh2} and \cite{os}, provide a complete description of length functions of a given finitely generated semigroup with enumerable set of relations inside a finitely presented semigroup.
\end{abstract}

\let\thefootnote\relax\footnotetext{{\bf Keywords:} Membership problem, Word problem, Embeddings of Semigroups, Length Function, Distortion.}
\let\thefootnote\relax\footnotetext{{\bf Mathematics Subject Classification 2000:} 20M05, 20F65.}

\section{Introduction}
\subsection{Preliminaries}
Let $S$ be an arbitrary semigroup (without signature identity element) with a finite generating set $\mathcal{A}=\{ a_1, \dots, a_m \}$. 

\begin{defn} The length of an element $g \in S$ is $|g|=|g|_{\mathcal{A}}$ is the length of the shortest word over the alphabet $\mathcal{A}$ which represents the element $g$, where for any word $W$ in $\mathcal{A}$ we define its length $||W||$ to be the number of letters in $W$.
\end{defn}

Observe that if the semigroup $S$ is embedded into another finitely generated semigroup $H$ with a generating system $\mathcal{B}=\{ b_1, \dots, b_k \}$, then for any $g \in S$ we have
\begin{equation}
|g|_{\mathcal{B}} \leq c|g|_{\mathcal{A}}
\label{one} 
\end{equation} with the constant $c = \max \{ |a_1|_{\mathcal{B}}, \dots,  |a_m|_{\mathcal{B}} \}$ independent of $g$. Motivated by inequality (\ref{one}), we introduce the following notion of equivalence.

\begin{defn}\label{equiv}
Let $l_1, l_2 : S \rightarrow \mathbb{N} = \{ 1, 2, 3, \dots \}$. We say that $l_1$ and $l_2$ are equivalent, $l_1 \approx l_2$, if there exist constants $c_1,c_2>0$ such that $$c_1l_1(g) \leq l_2(g) \leq c_2l_1(g)$$ for all $g \in S$. 
\end{defn}

The discussion above implies that the word length in $S$ does not depend up to equivalence on the choice of finite generating set.

We will also be considering a semigroup analogue of the notion of distortion. This idea was first introduced for groups by Gromov, in \cite{gromov}. We say that an embedding of one semigroup $H$ with finite generating system $\mathcal{B}$ into another semigroup $R$ with finite generating system $\mathcal{T}$ is undistorted if $$(|\cdot|_{\mathcal{T}}) \downharpoonright_H \approx |\cdot|_{\mathcal{B}}.$$ Otherwise, the embedding is distorted. The notion is clearly independent of the choice of finite generating sets $\mathcal{B}$ and $\mathcal{T}$.

\subsection{Statement of Main Results}

The main goal of this note is to prove an analog of Theorem 1 in \cite{olshanskii} for semigroups.
The necessary conditions for distortion functions of semigroups are as follows. The main result of this article is the sufficiency of said conditions.

\begin{lemma}\label{lm1}
Let $S$ be a semigroup and $l: S \rightarrow \mathbb{N}$ a function defined by some embedding of the semigroup $S$ into a semigroup $H$ with a finite generating system $\mathcal{B}=\{ b_1, \dots, b_k \}$; that is, $l(g) = |g|_{\mathcal{B}}$. Then
\begin{itemize}

\item [(D1)] $l(gh) \leq l(g)+l(h)$ for all $g,h \in S$;
\item [(D2)] There exists a positive number $a$ such that $\textrm{card} \{ g \in S : l(g) \leq r \} \leq a^r$ for any $r \in \mathbb{N}$.
\end{itemize}
\end{lemma}

\begin{proof}
The condition $(D1)$ is obvious. To prove the condition $(D2)$ it will suffice to take $a=k+1$. This follows because the number of all words in $\mathcal{B}$ having length $\leq r$ is not greater than $(k+1)^r$. 
\end{proof}

We establish the notation that the $(D)$ condition refers to conditions $(D1)$ and $(D2)$ of Lemma \ref{lm1}.

\begin{theorem}\label{t1}
\begin{enumerate}
\item For any semigroup and any function $l : S \rightarrow \mathbb{N}$ satisfying the $(D)$ condition, there is an embedding of $S$ into a $2$-generated semigroup $H$ with generating set $\mathcal{B} = \{ b_1, b_2 \}$, such that the function $g \rightarrow |g|_{\mathcal{B}}$ is equivalent to the function $l$.
\item For any semigroup $S$ and any function $l: S \rightarrow \mathbb{N}$ satisfying the $(D)$ condition, there is an embedding of $S$ into a finitely generated semigroup $K$ with finite generating set $\mathcal{C}$ such that the function $g \rightarrow |g|_{\mathcal{C}}$ is equal to the function $l$.
\end{enumerate}
\end{theorem}

\begin{cor}\label{c1}
\begin{enumerate}
\item Let $g$ be an element such that $g$ generates as infinite subsemigroup in a semigroup $H$ with finite generating set $\mathcal{B}=\{ b_1, \dots, b_k \}$;  i.e. $\textrm{card}\{g^n\}_{n \in \mathbb{N}}=\infty$. Denote $l(i)=|g^i|_{\mathcal{B}}=|g^i|$ for $i \in \mathbb{N}$. Then
\begin{itemize}

\item [(C1)] $l(i+j) \leq l(i)+l(j)$ for all $i,j \in \mathbb{N}$ ($l$ is subadditive);
\item [(C2)] There exists a positive number $a$ such that $\textrm{card} \{ i \in \mathbb{N} : l(i) \leq r \} \leq a^r$ for any $r \in \mathbb{N}$.
\end{itemize}
\item For any function $l : \mathbb{N} \rightarrow \mathbb{N}$, satisfying conditions $(C1)$ and $(C2)$, there is a $2$-generated semigroup $H$ and an element $h \in H$ such that $|h^i|_H \approx l(i)$.
\item For any function $l : \mathbb{N} \rightarrow \mathbb{N}$, satisfying conditions $(C1)$ and $(C2)$, there is a finitely generated semigroup $K$ and an element $k \in K$ such that $|k^i|_K = l(i)$.
\end{enumerate}
\end{cor}

We observe that the main result of \cite{olsh2} also holds for semigroups. 

\begin{theorem}\label{two}
Let $l$ be a computable function with properties $(D1)-(D2)$ on a semigroup $S$. Suppose further that $S$ has enumerable set of defining relations. Then $S$ can be isomorphically embedded into some finitely presented semigroup $R$ in such a way that the function $l$ is equivalent to the restriction of $|\textrm{  } |_R$ to $S$.
\end{theorem}

This Theorem will be proved in Section \ref{fourth}.

\begin{ex}
Because the function $l:\mathbb{N} \rightarrow \mathbb{N}: i \mapsto \lceil i^{\pi-e} \rceil$ is computable ($\pi$ and $e$ being computable numbers) and satisfies the $(D)$ condition, we have by Theorem \ref{two} that there exists a finitely presented semigroup $R$ and an element $r \in R$ such that $|r^i|_R \approx l(i)$.
\end{ex}

Theorem \ref{two} fails to provide a complete description of length functions of a given finitely generated semigroup with enumerable set of relations inside finitely presented semigroups. In \cite{os}, the corresponding question was answered for groups, by extending the $(D)$ condition. We obtain a semigroup analog of the main result in \cite{os} as follows. 

We use the notation that $F_m$ is an absolutely free semigroup of rank $m$. Given an $m$-generated semigroup $S$, and a function $l:S \rightarrow \mathbb{N}$, we may obtain the natural lift function $l^{*}:F_m \rightarrow \mathbb{N}.$

\begin{defn}
Let $S$ be an $m$-generated semigroup, and $l:S \rightarrow \mathbb{N}$. We say that $l$ satisfies condition $(D3)$ if there exists a natural number $n$ and a recursively enumerable set $T \subset F_m \times F_n$ such that 
\begin{enumerate}
\item $(v_1, u), (v_2, u) \in T$ for some words $v_1,v_2,u$ then $v_1$ and $v_2$ represent the same element in $S$.
\item If $v_1$ and $v_2$ represent the same element in $S$ then there exists an element $u$ such that $(v_1,u),(v_2,u) \in T$.
\item $l^{*}(v)=\min\{||u|| : (v,u) \in T\}$ for every $v \in F_m$.
\end{enumerate}
\end{defn}

\begin{theorem}\label{d4}
Let $S$ be a finitely generated subsemigroup of a finitely presented semigroup $H$. Then the corresponding length function on $S$ satisfies conditions $(D1)-(D3)$. Conversely, for every finitely generated semigroup $S$ and function $l:S \rightarrow \mathbb{N}$ satisfying conditions $(D1)-(D3)$, there exists an embedding of $S$ into a finitely presented semigroup $H$ such that the length function $g \rightarrow |g|_H$ is equivalent to $l$, in the sense of Definition \ref{equiv}.
\end{theorem}

This Theorem will be proved in Section \ref{fourth}.

When $S$ has solvable word problem, the condition $(D3)$ can be replaced by a simpler condition. 

\begin{defn}
The graph of a function $l^{*}:F_m \rightarrow \mathbb{N}$ is the set $\{(w,l^{*}(w)) : w \in F_m\}$. A pair $(w,k)$ is said to lie above the graph of $l^{*}$ if $l^{*}(w) \leq k$. 
\end{defn}

We observe that the following result of \cite{os} also holds in the semigroup setting. In fact, the proof uses no special properties of groups such as existence of identity element or inverses and so goes through immediately and directly.

\begin{theorem}\label{hib}
Let $S$ be a finitely generated semigroup with decidable word problem. Then the function $l:g \mapsto |g|_H$ given by an embedding of $S$ into a finitely presented semigroup $H$ satisfies the conditions $(D1)-(D2)$ as well as the following condition:\\

\vspace{.05in}
$(D3')$ The set of pairs above the graph of $l^{*}$ is recursively enumerable. \\

\vspace{.05in}
Conversely, for every function $l: S \rightarrow \mathbb{N}$ satisfying $(D1),(D2)$ and $(D3')$, there exists an embedding of $S$ into a finitely presented semigroup $H$ such that the corresponding length function on $S$ is equivalent (in the sense of Definition \ref{equiv}) to $l$.
\end{theorem}

The following Corollary follows from Theorem \ref{hib} and reminds us of the statement of Corollary \ref{c1}.

\begin{cor}
\begin{enumerate}
\item Let $g$ be an element generating an infinite subsemigroup in a finitely presented semigroup $H$ with generating set $\mathcal{B}=\{b_1, \dots, b_k\}$. Denote $l(i)=|g^{i}|_{\mathcal{B}}=|g^i|$ for $i \in \mathbb{N}$. Then
\begin{itemize}

\item [(C1)] $l(i+j) \leq l(i)+l(j)$ for all $i,j \in \mathbb{N}$ ($l$ is subadditive);
\item [(C2)] There exists a positive number $a$ such that $\textrm{card} \{ i \in \mathbb{N} : l(i) \leq r \} \leq a^r$ for any $r \in \mathbb{N}$.
\item [(C3)] The set of natural pairs above the graph of $l$ is recursively enumerable.
\end{itemize}
\item Conversely, For any function $l : \mathbb{N} \rightarrow \mathbb{N}$, satisfying conditions $(C1)-(C3)$, there is a finitely presented semigroup $H$ and an element $g \in H$ such that $|g^i|_H \approx l(i)$.
\end{enumerate}
\end{cor}

\section{Exponential Sets of Words}

\begin{defn}
Let $\mathcal{X}$ be a set of words over the alphabet \\ $\mathcal{A} = \{ a_1, \dots, a_m \}$. We call $\mathcal{X}$ exponential if there are constants $N$ and $c>1$ such that $$\textrm{card}\{X \in \mathcal{X} : ||X|| \leq i \} \geq c^i$$ for every $i \geq N$.
\end{defn}

\begin{defn}\label{d1}

A collection $\mathcal{Y}$ of words satisfies the overlap property if whenever $Y, Z \in \mathcal{Y}$ we have that 
\begin{equation}\label{e1}
Y \textrm{ is not a proper subword of } Z \textrm{ and }
\end{equation}
\begin{equation}\label{e2}
U \textrm{ nonempty, }Y \equiv UV \textrm{ and } Z \equiv WU \textrm{ implies }Y \equiv U \equiv Z
\end{equation}
where $\equiv$ represents letter-for-letter equality.
\end{defn}

\begin{lemma}\label{kul}
There exists an exponential set of words in the alphabet $\{b_1,b_2\}$ satisfying the overlap property of Definition \ref{d1}.
\end{lemma}

\begin{proof}
Consider the set $\mathcal{M}$ of all words $$\{ b_1^3Vb_2^3 : V \equiv b_2V'b_1 \textrm{ contains neither } b_1^3 \textrm{ nor } b_2^3 \textrm{ as a subword.}\}$$ This set does satisfiy the overlap property of Definition \ref{d1}. Condition (\ref{e1}) is satisfied because if $Y, Z \in \mathcal{M}$ and $Y$ is a subword of $Z \equiv W_1YW_2$, then we have that $b_1^3$ is a prefix of both $Y$ and $Z$. However, the only time that $b_1^3$ can occur in a word in $\mathcal{M}$ is at the very beginning. Therefore, $W_1$ is empty. Similarly, $W_2$ is empty. Condition (\ref{e2}) is satisfied because if $Y \equiv UV$ and $Z \equiv WU$ then the prefix of $U$ must be $b_1$ and the suffix of $U$ must be $b_2$, say $U=b_1U'b_2$. This implies that $b_1 U' b_2V \equiv Wb_1U'b_2 \equiv b_1^3V'b_2^3$ for some $V'$. Therefore, $U \equiv b_1^3 V''b_2^3$, for some $V''$ which implies that both $V$ and $W$ are empty.\\
We will verify that $\mathcal{M}$ is an exponential set. Consider the set $$M_i=\{x \in \mathcal{M} : ||x|| \leq i \}.$$ Consider a word $x \equiv b_1^3 b_2 b_2^{\beta_1} b_1^{\alpha_1} b_2^{\alpha_2} \cdots b_2^{\alpha_n}b_1^{\beta_2}b_1b_2^3$ where $\beta_j \in \{0,1\}$ for $j=1,2$ and $\alpha_j \in \{1,2\}$ for $j \in \{1, \dots, n\}$, and $n = \frac{i-10}{2}$. Such a word has $||x|| \leq 10+2n = i$ so $x \in M_i$. If $i>N=12$, then there exists $c>1$ satisfying $2^{\frac{i-6}{2i}}>c.$ This implies that $\textrm{card}(M_i) \geq 2^{\frac{i-6}{2}} \geq c^i$ for all $i \geq N$. 
\end{proof}

 \begin{lemma}\label{lo}
 Let $\mathcal{M}$ be an exponential set satisfying the overlap property. Suppose $V \equiv X_1X_2 \cdots X_t \equiv SY_1 Y_2 \cdots Y_mT$ where $m,t \geq 1$ and $X_n,Y_j \in \mathcal{M}$ for all $1 \leq n \leq t, 1 \leq j \leq m$. Then there exists an $i \leq t$ such that $S \equiv X_1 \cdots X_{i-1}, T \equiv X_{i+m} \cdots X_t$ and $Y_j \equiv X_{-1+i+j}$ for $j=1, \dots, m$.
 \end{lemma}
 
 \begin{proof}
 Because $X_1X_2 \cdots X_t \equiv SY_1 Y_2 \cdots Y_mT$ is letter-for-letter equality, we know that the first letter, $u$, in $Y_1$ also occurs in $X_i$ for some $i$. We proceed by considering cases. If $u$ is also the first letter in $X_i$ then either $X_i$ is a subword of $Y_1$ or vice-versa. In either of these cases, by condition (\ref{e1}), we have that $X_i \equiv Y_1$. Now suppose that $u$ is not the first letter in $X_i$. If $Y_1$ is a subword of $X_i$ then we apply condition (\ref{e1}) again. Otherwise, a suffix of $X_i$ must equal a prefix of $Y_1$, which implies by condition (\ref{e2}) that $X_i \equiv Y_1$. Now consider $Y_2$. We know that the first letter of $Y_2$ must also be the first letter of $X_{i+1}$. Therefore, one is a subword of the other, so by condition (\ref{e1}) we obtain that $Y_2 \equiv X_{i+1}$. The same argument shows that $Y_j \equiv X_{-1+i+j}$ for $j=1, \dots, m$ hence $T \equiv X_{i+m} \cdots X_t$ and $S \equiv X_1 \cdots X_{i-1}$.
 \end{proof}

\begin{lemma}\label{li}
Let $\mathcal{M}$ be an exponential set of words  over a finite alphabet $\{ a_1, \dots, a_m \}$. Then for a given function $l: S \rightarrow \mathbb{N}$ satisfying the $(D)$ condition, there is a constant $d=d(\mathcal{M},l)$ such that there exists an injection $S \rightarrow \mathcal{M}: g \mapsto X_g \in \mathcal{M}$ satisfying 
\begin{equation}\label{six}
l(g) \leq ||X_g|| < dl(g), g \in S.
\end{equation}
\end{lemma}

\begin{proof}
A proof can be found in \cite{olshanskii} for the case where words are considered in a positive alphabet and hence it holds for semigroups as well. 
\end{proof}

\section{Constructing the Embedding}\label{constructing}

We begin by fixing some notation. Let $\mathcal{M}$ be the exponential set of words in the alphabet $\mathcal{B}=\{b_1, b_2\}$ obtained in Lemma \ref{kul}. Let $S$ be a semigroup and $l:S \rightarrow \mathbb{N}$ a function satisfying the $(D)$ condition. Let $d=d(\mathcal{M},l)$ and $\mathcal{X}=\{X_g\}_{g \in S} \subset \mathcal{M}$ be the constant and exponential subset guaranteed by Lemma \ref{li} and satisfying the inequality (\ref{six}).

The semigroup $S$ is a homomorphic image of the free semigroup $F_S$ with basis $\mathcal{A}=\{x_g\}_{g \in S}$ under the epimorphism $\varepsilon: x_g \mapsto g$. Let $\rho = \ker(\varepsilon)$. Therefore, $S \cong F_S/\rho$, and $\rho$ provides all relations which hold in $S$. Let $$R=\{(x_h,x_{h'}x_{h''}):h=h'h'' \textrm{ in } S\}.$$ Then $R$ represents the relations of $S$ arising from its multiplication table. The proof of the following Lemma is elementary so we omit it.

\begin{lemma}\label{lp}
The semigroup $S$ has presentation $\langle \mathcal{A} | R \rangle$.
\end{lemma}

Consider the following commutative diagram:
\[ \xymatrix{ F_S \ar@{->}[rr]^{\displaystyle \varepsilon(x_g)=g} \ar@{->}[dd]_{\displaystyle \beta(x_g)=X_g}      && S \ar@{->}[dd]^{\displaystyle \gamma(g)=X_g\xi}    \\ \\ F(b_1,b_2) \ar@{->}[rr]_{\displaystyle \overline{\varepsilon}} && H=F(b_1,b_2)/\xi } \]

\noindent where $\xi$ is the unique smallest congruence relation on the free semigroup $F(b_1,b_2)$ containing the set $\beta R = \{(X_h,X_{h'}X_{h''}) : h=h'h'' \textrm{ in } S\}$ of defining relations of $H$, and $ \overline{\varepsilon}$ is the natural epimorphism. Observe that $\gamma$ may be well-defined by the formula $\gamma \varepsilon = \overline{\varepsilon}\beta$; i.e. $\gamma := \overline{\varepsilon}\beta \varepsilon^{-1}$. This definition is independent of the choice of $\varepsilon^{-1}(g)$ for $g \in S$; in particular, we may select representative $\varepsilon^{-1}(g)=x_g.$  This is because if we have two representatives, $\varepsilon^{-1}(g)= x_g = x_{g_1} x_{g_2} \cdots x_{g_n}$ then $\varepsilon(x_g)=\varepsilon(x_{g_1}x_{g_2} \cdots x_{g_n})=\varepsilon(x_{g_1}) \cdots \varepsilon(x_{g_n})$ so $g = g_1 \cdots g_n$ in $S$. One computes that $\overline{\varepsilon}\beta (x_{g_1}x_{g_2} \cdots x_{g_n}) = \overline{\varepsilon} (X_{g_1}X_{g_2} \cdots X_{g_n}) = X_{g_1}X_{g_2}\cdots X_{g_n} \xi$ and $\overline{\varepsilon}\beta (x_{g}) = \overline{\varepsilon}X_{g}=X_g \xi$. By definition, $X_{g_1}X_{g_2} \cdots X_{g_n} \xi = X_g \xi$ if $(X_{g_1}X_{g_2}\cdots X_{g_n},X_g) \in \xi$. By induction, we may assume that \\ $(X_{g_1} \cdots X_{g_{n-1}}, X_{g_1 \cdots g_{n-1}}) \in \xi$. Then because $\xi$ is left compatible, we have that $(X_{g_1} \cdots X_{g_n}, X_{g_1 \cdots g_{n-1}}X_{g_n}) \in \xi$. By definition of $\beta R$ we also have that $(X_{g_1 \cdots g_{n-1}}X_{g_n}, X_g) \in \xi$. Therefore, $(X_{g_1} \cdots X_{g_n}, X_g) \in \xi$ as required.

\begin{lemma}\label{di}
The map $\beta$ is injective.
\end{lemma}

\begin{proof}
Suppose $x_{g_1} \cdots x_{g_n}, x_{h_1}, \cdots x_{h_m} \in F_S$ and \\ $\beta(x_{g_1} \cdots x_{g_n})=\beta(x_{h_1}, \cdots x_{h_m})$. Then $X_{g_1} \cdots X_{g_n} = X_{h_1} \cdots X_{h_m}$. Because $X_{g_1} \cdots X_{g_n}$ and $X_{h_1} \cdots X_{h_m}$ are words in the free group $F(b_1,b_2)$ the equality must in fact be letter-for-letter. Therefore by Lemma \ref{lo}, we have that $n=m$ and $X_{g_i} \equiv X_{h_i}$ for $i=1, \dots, n$. By Lemma \ref{li} the map $S \rightarrow F(b_1,b_2) : g \mapsto X_g$ is injective, hence $g_1 = h_1, \dots, g_n=h_n$ so $x_{g_1} \cdots x_{g_n} = x_{h_1}, \cdots x_{h_m}$. 
 
\end{proof}

\begin{lemma}\label{lgi}
The map $\gamma$ is injective.
\end{lemma}

\begin{proof}
Suppose that $g,g' \in S$ and $\gamma (g)=\gamma (g')$. We will show that $g=g'$. Since $\gamma=\overline{\varepsilon}\beta \varepsilon^{-1}$, we have that $\overline{\varepsilon}\beta x_g = \overline{\varepsilon}\beta x_{g'}$ which implies that $\overline{\varepsilon}X_g =\overline{\varepsilon}X_{g'}$. Thus by definition of $\overline{\varepsilon}$ we have that $(X_g,X_{g'}) \in \xi$ which means that there is a finite chain $$X_g=X_{k_0} \rightarrow X_{k_1} \rightarrow X_{k_2} \cdots \rightarrow X_{k_m}=X_{g'}$$ where each $\rightarrow$ is obtained by applying a defining relation. Every $X_{k_i}$ is a product of elements of the form $X_h$ where $h \in S$. Each time we apply a defining relation, we replace one $X_h$ with $X_{h'}X_{h''}$ or vice-versa, where $h=h'h''$ in $S$. Therefore, for each $X_{k_i}$, the product of subscripts equals the same element of $S$; in particular, $g=g'$ as required.
\end{proof}

Let $H_S$ be the free subsemigroup of $F(b_1,b_2)$ with free generating set $\{X_g\}_{g \in S}$. We know that $H_S$ is free by Lemma \ref{di}, because $H_S = \textrm{im}  \beta \cong F_S / \ker \beta \cong F_S$. As $ \overline{\varepsilon}$ is an epimorphism, we can consider the system $\mathcal{B}=\{b_1, b_2\}$ to be a generating set for the semigroup $H$ which contains the isomorphic copy $\gamma(S)$ of $S$, by Lemma \ref{lgi}.

By an $H_S$-word we mean any word of the form $W(X_g, \dots, X_h)$. Any $H_S$-word can be rewritten as a word in the letters $b_1$ and $b_2$.

The following is an important ingredient in the proof of Theorem \ref{t1} Part $(1)$.

\begin{lemma}\label{lw}
For any $H_S$-word $U$, there is an $H_S$-word $V$ such that $\overline{\varepsilon}(V)=\overline{\varepsilon}(U)$ and $||V|| \leq ||W||$ for any word $W$ with $\overline{\varepsilon}(W)=\overline{\varepsilon}(U)$.
\end{lemma}

\begin{proof}
It suffices to show that if a word $W$ satisfies $\overline{\varepsilon}(W)=\overline{\varepsilon}(U)$ then $W$ must be an $H_S$-word. 
Because $W=U$ in $H$ there is a finite chain $$U=U_0 \rightarrow U_1 \rightarrow \cdots \rightarrow U_m=W$$ where each $\rightarrow$ is obtained by applying a defining relation in $H$. Suppose by induction that at the $n^{th}$ step we have $U_{n} \rightarrow U_{n+1}$ where the $H_S$-word $U_n \equiv X_{l_1} \cdots X_{l_t}$ for $l_1, \dots, l_t \in S$. Therefore we have that $X_{l_1} \cdots X_{l_t} \equiv T'X_hT = T'X_{h'}X_{h''}T \equiv U_{n+1}$ for some words $T,T'$ where the defining relation applied was $X_h=X_{h'}X_{h''}$ for $h=h'h''$ in $S$. By Lemma \ref{lo}, both $T$ and $T'$ are $H_S$-words. Thus so is $U_{n+1}$, and by induction, $W$. 
\end{proof}

\begin{proof} of Theorem \ref{t1} Part $1$: \\
By Lemma \ref{lgi} we may identify $S$ with its image $\gamma(S) \subset H$. The equalities $$g = \gamma (g) = \overline{\varepsilon} \beta \varepsilon^{-1}(g) = \overline{\varepsilon} \beta(x_g) = \overline{\varepsilon}(X_g)$$ and the inequalities (\ref{six}) yield 
\begin{equation}\label{e3}
|g|_{\mathcal{B}} \leq dl(g)
\end{equation}
for $d>0$ and for any $g \in S \subset H$.
To obtain the opposite estimate, we consider an element $g \in S$ and apply Lemma \ref{lw} to the $H_S$-word $U \equiv X_g$. For a word $W$ of minimum length representing the element $X_g$, and for the $H_S$-word $V$ from Lemma \ref{lw}, we have 
\begin{equation}\label{e4}
|g|_{\mathcal{B}}=||W|| \geq ||V||.
\end{equation}
By definition of $H_S$ there exists a unique decomposition of the $H_S$-word $V$ as a product $V \equiv X_{g_1} X_{g_2} \cdots X_{g_s}$ for some $g_j \in S$. 
Because $V=W$ in $H$ we have that $(X_{g_1} \cdots X_{g_s},X_g) \in \xi$ which implies by previous arguments that $g = g_1 \cdots g_s$ in the subsemigroup $S$ of $H$. Taking into account the inequalities (\ref{six}) we conclude that $||X_{g_j}|| \geq l(g_j)$. Hence, by the condition $(D_1)$ we have that $$||V|| =\displaystyle\sum_{j=1}^s||X_{g_j}|| \geq \displaystyle\sum_{j=1}^s l(g_j) \geq l(g).$$ Therefore, $|g|_{\mathcal{B}} \geq l(g)$, by (\ref{e4}). This, together with inequality (\ref{e3}), completes the proof.
\end{proof}

The following Lemma will essentially prove Theorem \ref{t1} Part $2$. We fix notation as in the Theorem: $S$ is a finitely generated semigroup, and $l:S \rightarrow \mathbb{N}$ satisfies the $(D)$ condition.

\begin{lemma}\label{ewo}
There is an exponential set of words $\mathcal{N}$ over a finite alphabet $\mathcal{C}$ satisfying the overlap property such that there is an injection $S \rightarrow \mathcal{N} : g \mapsto X_g$ satisfying
\begin{equation}\label{e5}
l(g)=||X_g||.
\end{equation}
\end{lemma}

\begin{proof}
Let $a$ be the integer arising from condition $(D2)$ for the given function $l$. Let $\mathcal{C}=\{c_1, \dots, c_{a+2}\}$. It suffices to produce a set of words $\mathcal{N}$ satisfying the overlap property and subject to $$\mathrm{card}\{y \in \mathcal{N} : ||y|| = i \} \geq a^i.$$ For if this is satisfied, then for every $g \in S$ we may find a distinct word of length $l(g)$ from our exponential set satisfying the overlap property. 
The same argument as that given in the proof of Lemma \ref{kul} shows that the set $$\mathcal{N} =\{c_1v(c_2, \dots, c_{a+1})c_{a+2}\}$$ where $v$ is an arbitrary word in $c_2, \dots, c_{a+1}$ does satisfy the required properties.
\end{proof}

\begin{remark}
Observe that Theorem \ref{t1} Part $2$ follows from Lemma \ref{ewo} by replacing the set $\mathcal{M}$ by $\mathcal{N}$ and the inequalities (\ref{six}) by equality (\ref{e5}) everywhere in the proof of Theorem \ref{t1} Part $1$.
\end{remark}

\section{Embedding to Finitely Presented Semigroups}\label{fourth}
In this section we will prove Theorems \ref{two} and \ref{d4}.

We begin with an undistorted analogue of Murskii's embedding theorem.

\begin{theorem}\label{three}
Let $H$ be a semigroup with a finite generating set $\mathcal{B}$ and a recursively enumerable set of (defining) relations. Then there exists an isomorphic embedding of $H$ in some finitely presented semigroup $R$ with generating set $\mathcal{T}$ without distortion.

\end{theorem}

Observe that Theorem \ref{two} follows immediately from Theorem \ref{t1}, Part $1$, Theorem \ref{three} and the assumption that $S$ has recursively enumerable set of defining relations.

Although an undistorted semigroup analog of Murskii's embedding appears in \cite{birget}, that Theorem makes additional assumptions regarding time complexity of the word problem in $H$. It is not clear to the author whether a simple proof of Theorem \ref{three} may be extracted from \cite{birget}.

To prove Theorem \ref{three} we will instead use such an embedding which was invented in \cite{murskii}, and show that it is undistorted.

\begin{proof} of Theorem \ref{three}. 
Let $P \in H$ and $W$ is a word representing the image of $P$ in $R$ under the embedding. We have by \cite{murskii} Lemma 3.3 that if a word $P$ in the alphabet $\mathcal{B}$ is equal in $R$ to a word $W$ in the alphabet $\mathcal{T}$ then it is possible to represent $W$ in the form $$W \equiv P_0U_1P_1U_2 \cdots U_lP_l$$ such that
\begin{enumerate}
\item All $P_i$'s are words in the alphabet $\mathcal{B}$;
\item One can delete some subwords from every $U_i$ and obtain a word $U'_i$, which by Lemma 3.1 in \cite{murskii} has subword $\tilde{R_i}$, where $R_i$ are words in the alphabet $\mathcal{B}$ and $\tilde{R_i}$ is are not words in the alphabet $\mathcal{B}$, but $||\tilde{R_i}||=||R_i||$ for all $i$. 

\item The word $P_0R_1P_1 \cdots R_lP_l$ is equal to $P$ in $H$.
\end{enumerate}

This implies that $$|P|_\mathcal{B} \leq |P_0R_1P_1 \cdots R_lP_l|_\mathcal{B} \leq |P_0|_\mathcal{B}+|R_1|_\mathcal{B}+ \cdots + |P_l|_\mathcal{B}$$ $$ \leq ||P_0||+||R_1||+ \cdots + ||P_l||=||P_0||+||\tilde{R_1}||+ \cdots + ||P_l|| $$ $$\leq ||P_0||+||U_1'||+ \cdots + ||P_l|| \leq ||P_0||+||U_1||+ \cdots + ||P_l|| = ||W||.$$

Indeed, we have that $||\tilde{R_i}|| \leq ||U_i'||$ because $\tilde{R_i}$ is a subword of $U_i'$ for all $i$. Similarly, because $U_i'$ is obtained from $U_i$ by deleting subwords, we have $|U_i| \geq |U_i'|$ for all $i$. Since $W$ is any word equal to $P$ in $R$, the above inequalities hold in particular when $||W||=|P|_{\mathcal{T}}$ so we have that $|P|_\mathcal{B} \leq |P|_{\mathcal{T}}$, which shows that the embedding is undistorted.
 \end{proof}

We proceed with consideration of Theorem \ref{d4}, in particular towards establishing notation to be used in the proof.

Let $S$ be a finitely generated semigroup with generating set $\mathcal{A}=\{a_1, \dots, a_m\}$. For any $k>0$, let $F_k$ denote the free semigroup of rank $k$. 

Suppose that a function $l:S \rightarrow \mathbb{N}$ satisfies conditions $(D1)-(D3)$. Let $\pi: F_m \rightarrow S$ be the natural projection. By hypothesis, there exists a recursively enumerable set $T$ satisfying Properties $(1), (2),$ and $(3)$ of Condition $(D3)$. Let $U$ be the natural projection of $T$ onto $F_n$. Let $\phi: U \rightarrow F_m$ such that $v=\phi(u)$ if $(v,u) \in T$ and $(v,u)$ is the first pair in the enumeration of $T$ whose second component is $u$. 

By Lemma \ref{kul} there exists an exponential set of words $\mathcal{M}$ over the alphabet $\{x_1,x_2\}$ satisfying the overlap condition of Definition \ref{d1}. For the word length function $F_n \rightarrow \mathbb{N}$, there exists by Lemma \ref{li} a constant $d$ and an injection $\psi: F_n \rightarrow \mathcal{M} \subset F_2=F(x_1,x_2) : u \rightarrow X_u$ satisfying 
\begin{equation}\label{e6}
||u|| \leq ||X_u|| <d||u||.
\end{equation}
We may chose the function $\psi$ to be recursive. Begin by putting an order (e.g. ShortLex) on $U$. Then for every $u$ starting with the shortest we select the smallest word $X_u$ satisfying equation (\ref{e6}) and such that $X_u \ne X_{u'}$ if $u'<u$. 

Let $F(V)$ be the free semigroup with basis $V=\{x_v\}_{v \in F_m}$. Consider the natural epimorphism defined on generators by $\zeta: F(V) \rightarrow S: \zeta(x_v)=\pi(v)$. Define the free semigroup $F(Y)$ with basis $Y=\{y_u\}_{u \in U}$. Let $\eta: F(Y) \rightarrow F(V)$ be defined by $\eta(y_u)=x_{\phi(u)}.$ Then the product $\varepsilon = \zeta \eta$ is an epimorphism because by Parts $(1)$ and $(2)$ of Condition $(D3)$, for any $v \in F_m$ there is $(v',u) \in T$ such that $\phi(u)=v'$ and $\pi(v')=\pi(v).$ Therefore, there is a presentation $S=\langle Y | \mathcal{R} \rangle$ defined by the isomorphism $S \cong F(Y)/\ker(\varepsilon).$

Define a homomorphism $\beta: F(Y) \rightarrow F_2: \beta(y_u)=\psi(u)=X_u$. Let $\xi$ be the unique smallest congruence relation on the free semigroup $F_2$ containing the set $\beta(\mathcal{R})=\{(\beta(a),\beta(b)):(a,b) \in \mathcal{R}\}$. Let $\overline{\varepsilon}$ the natural epimorphism of $F_2$ onto $H=F_2/\xi$. Let $\gamma: S \rightarrow H$ be defined by $\gamma=\overline{\varepsilon}\beta\varepsilon^{-1}$. There is also a map $F(V) \rightarrow F_m: x_v \mapsto v.$
Consider the commutative diagram defined by all these maps:

\[ \xymatrix{
			&	   F(V) \ar[dr]^{\zeta} \ar[r] & 		      F_m \ar[d]^{\pi} \\
& F(Y)  \ar[d]^{\beta} \ar[r]^{\varepsilon} \ar[u]^{\eta} & S \ar[d]^{\gamma} \\
			  & F_2 \ar[r]^{\overline{\varepsilon}}      & H} \]

\begin{lemma}\label{ert}
The map $\beta$ is injective.
\end{lemma}

\begin{proof}
This fact is proved exactly similarly to Lemma \ref{di}. The application of Lemma \ref{lo} is still valid, because our set $\mathcal{M} \supset \{X_u\}_{u \in U}$ is exponential and satisfies the overlap property. Moreover, we have that the map $U \rightarrow F_2: u \rightarrow X_u$ is injective. These are the only facts used in the proof of Lemma \ref{di}.
\end{proof}

\begin{lemma}\label{gm}
The map $\gamma$ is a well-defined monomorphism. 
\end{lemma}
\begin{proof}
The fact that $\gamma$ does not depend on the choice of preimage under $\varepsilon$ of $g \in S$ follows exactly as the proof of the same fact in Section \ref{constructing}. Moreover, $\gamma$ is injective. The proof is similar to that of Lemma \ref{lgi}. 
\end{proof}

\begin{lemma}\label{recur}
The semigroup $H$ is recursively presented.
\end{lemma}

\begin{proof}
The set of defining relations for $H$ is $\xi=\beta(\mathcal{R})$. Because the map $\psi: F_n \rightarrow F_2$ was chosen to be recursive, and by definition of $\beta$, it suffices to show that the relations $\mathcal{R}$ are recursively enumerable. First observe that the set of relations of $S$ in generators $\{a_1, \dots, a_m\}$ is recursively enumerable. We have by Condition $(D3)$, Parts $(1)$ and $(2)$ that $v_1(a_1, \dots, a_m)=v_2(a_1, \dots, a_m)$ in $S$ if and only if 
both $(v_1,u),(v_2,u) \in T$ for some $u$. Therefore, because $T$ is recursively enumerable, so is the set of relations of $S$. 
Then we have that $(w(y_{u_1}, \dots, y_{u_s}),w'(y_{u_1'}, \dots, y_{u_t'})) \in \ker(\varepsilon)$ if and only if 
\begin{equation}\label{e7}
\zeta w(x_{\phi u_1}, \dots, x_{ \phi u_s})=\zeta w'(x_{\phi u_1'}, \dots, x_{\phi u_t'}) \textrm{ in } S.
\end{equation}
 Thus we have to enumerate such pairs $(w,w')$. To do this, we enumerate all variables of the form $x_{\phi u}$ with $u \in U$. This is possible by definition of $\phi$ and $U$ and by the fact that $T$ is recursively enumerable. Next, we enumerate all pairs $(w(x_{v_1}, \dots, x_{v_s}), w'(x_{v_1'}, \dots, x_{v_t'}))$ with $\zeta(w)=\zeta(w')$. This is possible because $\zeta(w)=\zeta(w')$ if and only if $\pi (w(v_1, \dots, v_s)) = \pi(w'(v_1', \dots, v_t'))$ if and only if 
$$ w(v_1(a_1, \dots, a_m), \dots, v_s(a_1, \dots, a_m)) =$$
 \begin{equation}\label{e8}
 w'(v_1'(a_1, \dots, a_m), \dots, v_t'(a_1, \dots, a_m)).
 \end{equation}
  We have already seen that the set of all relations of $S$ is recursively enumerable. Given any relation in $S$ in generators $\{a_1, \dots, a_m\}$, we may find all possible $w,w',v_1, \dots, v_t'$ such that the relation may be presented as it is written in equation (\ref{e8}). There is an algorithm which can do this because the lengths of possible $w,w',v_1, \dots, v_t'$ are bounded by the length of the given relation of $S$. To complete the proof it suffices to compare these two lists to obtain a list of all pairs $(w,w')$ satisfying equation (\ref{e7}).
\end{proof}

\begin{proof} of Theorem \ref{d4}.\\

We first suppose that $S$ is a semigroup with finite generating set $\mathcal{A}=\{a_1, \dots, a_m\}$ and that a function $l: S \rightarrow \mathbb{N}$ satisfies conditions $(D1)-(D3)$. Lemmas \ref{gm} and \ref{recur} show that there is an embedding $S \rightarrow H$ to a recursively presented and $2$-generated semigroup. By Theorem \ref{three}, it suffices to prove that the function $l:S \rightarrow \mathbb{N}$ is equivalent to the word length on $H$ restricted to $S$. Let $g = \pi(v) \in S$. By Part $(3)$ of Condition $(D3)$, there exists a word $u \in F_n$ such that $||u||=l(g)$. Let $v'=\phi(u)$. We have that $\pi(v)=\pi(v')$ by Part $(1)$ of Condition $(D3)$ and by the definition of $\phi$. Then $\varepsilon(y_u)=\pi(\phi(u))=\pi(v')=\pi(v)=g.$ Therefore, by definition we have that $\gamma(g)=\overline{\varepsilon}\beta(y_u)=\overline{\varepsilon}(X_u)$, and so 
\begin{equation}\label{five}
|\gamma(g)|_H \leq ||X_u|| \leq d||u||=dl(g).
\end{equation}
The reverse inequality follows exactly from the arguments of the Proof of Theorem \ref{t1} Part $(1)$, which only uses the overlap property, Lemma \ref{ert} and the replacing of inequalities $(\ref{six})$ by $(\ref{five})$ and the definition of $H_S$ by the free semigroup $\{X_u\}_{u \in U}.$

To prove the converse, suppose that $S$ is a subsemigroup of $H$ with generating set $\mathcal{B}=\{b_1, \dots, b_m\}$. We must show that $$l:S \rightarrow \mathbb{N}: l(g) = |g|_{\mathcal{B}}$$ satisfies condition $(D3)$. Since $H$ is finitely presented, the collection $T \subset F_m \times F_n$ defined by $$T=\{ (v,u) : v(a_1, \dots, a_m)=u(b_1, \dots, b_n) \textrm{ in } H\}$$ is recursively enumerable. Condition $(D3)$ Part $(1)$ is satisfied because if \\
$(v_1,u),(v_2,u) \in T$ then $v_1(a_1, \dots, a_m)=u(b_1, \dots, b_n)$ in $H$, and \\ $v_2(a_1, \dots, a_m)=u(b_1, \dots, b_n)$ in $H$. Therefore, since the map $S \rightarrow H$ is an injection, we have that $v_1(a_1, \dots, a_m)=v_2(a_1, \dots, a_m)$ in $S$. To see that Condition $(D3)$, Part $(2)$ is satisfied, suppose $v_1=v_2$ in $S$ and let $v_1(a_1, \dots, a_m) \in H$. Then we may write $v_1$ with respect to the generating set $\mathcal{B}$ of $H$; that is, there exists $u \in H$ with $v_1(a_1, \dots, a_m)=u(b_1, \dots, b_n)$. Now consider words $u(x_1, \dots, x_n) \in F_n, v_1(y_1, \dots, y_m) \in F_m$, where $F_n$ has basis $\{x_1, \dots, x_n\}$ and $F_m$ has basis $\{y_1, \dots, y_n\}$.  We have that $(v_1,u),(v_2,u) \in T$ because $u(b_1, \dots, b_n)=v_1(a_1, \dots, a_m)=v_2(a_1, \dots, a_m)$. To see that condition $(D3)$ Part $(3)$ is satisfied, let $v=v(y_1, \dots, y_m) \in F_m$. Then $$l^{*}(v)=l(v(a_1, \dots, a_m))=|v(a_1, \dots, a_m)|_{\mathcal{B}}$$  $$ =\min\{||u|| : u=v \textrm{ in } H\}=\min\{||u||: (v,u) \in T\}.$$
\end{proof}

\subsection*{Acknowledgements}
The author wishes to express gratitude to Alexander Olshanskii for his guidance and many valuable suggestions.

\noindent Tara C. Davis\\
Department of Mathematics\\
Vanderbilt University\\
1326 Stevenson Center\\
Nashville, Tennessee 37240, United States of America\\
tara.c.davis@vanderbilt.edu\\


\begin{thebibliography}{3}

\bibitem[M]{murskii} Murskii, V. L., Isomorphic Embeddability of Semigroups with Countable Sets of Defining Relations in Finitely Defined Semigroups, Mathematical Notes, 1, 2, 145-149, 1967.

\bibitem[B]{birget} Birget, J-C., Time Complexity of the Word Problem for Semigroups and the Higman Embedding Theorem, International Journal of Algebra and Computation, v8, 2, 235-294, 1998.

\bibitem[G]{gromov} Gromov, M., Geometric Group Theory: Asymptotic Invariants of Infinite Groups, London Mathematical Society Lecture Notes, Series 182, Cambridge University Press, 1993.

\bibitem[O]{olshanskii} Olshanskii, A. Yu., Distortion Functions for Subgroups, Geometric Group Theory Down Under, Ed. Cossey, J, Miller, C. F. III, Neumann, W. D., Shapiro, M., Proceedings of a Special Year in Geometric Group Theory, Canberra, Australia, 1996, Walter de Gruyter, Berlin, New York, 281-291, 1999.

\bibitem[O2]{olsh2} Olshanskii, A. Yu., On Subgroup Distortion in Finitely Presented Groups (in Russian), Sbornik Mathematics, 188, 11, 51-98, 1997; translation in Sbornik Mathematics, 188, 11, 1617-1664, 1997.

\bibitem[OS]{os} Olshanskii, A. Yu., Sapir, M. V., Length Functions on Subgroups in Finitely Presented Groups, Ed. Baik, Y. G., Johnson, D. L., Kim, A. C., Groups Korea '98, Proceedings of the International Conference, Pusan, Korea, 1998, Walter de Gruyter, Berlin, New York, 297-304, 2000.
\end{thebibliography}
\end{document}